\documentclass[12pt]{amsart}
\usepackage{amsmath,amssymb,amsfonts,mathrsfs,mathtools}
\usepackage{enumitem,xcolor,booktabs,array,needspace,longtable}
\usepackage[expansion=false]{microtype}
\usepackage{hyperref}
\usepackage[alphabetic,msc-links]{amsrefs}

\hypersetup{
 colorlinks=false,pdfpagemode=UseNone,
 pdfauthor={Xiaodong Wang and Bo Zhu},
 pdfborder={0 0 1},pdfborderstyle={/S/D/D[3 2]},
 citebordercolor={0 0.55 0},linkbordercolor={0 0.55 0},
 urlbordercolor={0 0.55 0},
 pdftitle={Volume Entropy and Scalar Curvature near a Hyperbolic Metric}}
\usepackage{geometry}
\newtheorem{theorem}{Theorem}[section]
\newtheorem{proposition}[theorem]{Proposition}
\newtheorem{lemma}[theorem]{Lemma}
\newtheorem{corollary}[theorem]{Corollary}
\newtheorem{question}[theorem]{Question}
\theoremstyle{definition}

\newtheorem{remark}[theorem]{Remark}
\numberwithin{equation}{section}
\newcommand{\Met}{\mathcal M}
\newcommand{\TT}{\mathcal T_{\mathrm{TT}}}
\newcommand{\Vol}{\operatorname{Vol}}
\newcommand{\Ric}{\operatorname{Ric}}
\newcommand{\Sc}{\operatorname{Sc}}
\newcommand{\Cov}{\operatorname{Cov}}
\newcommand{\Var}{\operatorname{Var}}
\newcommand{\tr}{\operatorname{tr}}
\newcommand{\Hess}{\operatorname{Hess}}

\newcommand{\cH}{\mathcal H}
\newcommand{\ip}[2]{\langle #1,#2\rangle}
\begin{document}
\title[Volume entropy and scalar curvature]
{Volume Entropy and Scalar Curvature\\
near a Hyperbolic Metric}
\author{Xiaodong Wang}
\address{Department of Mathematics,
Michigan State University, East Lansing, MI 48824, USA}
\email{xwang@math.msu.edu}
\author{Bo Zhu}
\address{Yau Mathematical Sciences Center, Tsinghua University, Beijing, China}
\email{zhub@tsinghua.edu.cn}
\thanks{Bo Zhu is supported by NSFC~12501066.}
\date{October 3, 2026}
\begin{abstract}
Let $(M^n,g_0)$ be a closed hyperbolic manifold, $n\geq3$.
We compute the Hessian of volume entropy at $g_0$ under the constraint
$\Sc_g=-n(n-1)$ and express it as a quadratic form on transverse-traceless
symmetric two-tensors. In dimension three, the Hessian is negative
definite, and every sufficiently $C^\infty$-close metric with
$\Sc_g\geq-6$ has entropy at most $2$, with equality precisely for
pullbacks of $g_0$. In dimensions $n\geq4$, a universal spectral
threshold determines the Hessian sign. Nonzero trace-free Codazzi
tensors give positive directions, and hyperbolic bending provides
smooth constrained entropy saddles in every such dimension.
\end{abstract}
\maketitle
\section{Introduction}\label{sec:introduction}
Let $(M^n,g)$ be a closed connected Riemannian manifold, and let
$(\widetilde M,\widetilde g)$ be its universal Riemannian cover.
The volume entropy of $g$ is defined by
\begin{equation*}
 h(g)=\lim_{R\to\infty}\frac1R
 \log\Vol_{\widetilde g}\bigl(B_{\widetilde g}(\widetilde x,R)\bigr).
\end{equation*}
The limit exists, is independent of the basepoint, and equals the
exponential growth rate of a deck orbit. For negatively curved metrics,
it agrees with the topological entropy of the geodesic flow
(see \cite{Manning}).

Suppose that $M$ admits a hyperbolic metric $g_0$ of sectional
curvature $-1$. Then $h(g_0)=n-1$. Bishop--Gromov comparison gives
$h(g)\leq n-1$ whenever $\Ric_g\geq-(n-1)g$, and equality holds
precisely when $g$ is hyperbolic (see \cite{LW}*{Theorem~2}).
Does the same entropy bound hold under a scalar-curvature lower bound?
\begin{equation*}
 \Sc_g\geq-n(n-1)\quad\Longrightarrow\quad h(g)\leq n-1?
\end{equation*}
Together with the minimal-entropy theorem of Besson--Courtois--Gallot
(see \cite{BCG}), this would imply
\begin{equation*}
 \Sc_g\geq-n(n-1)\quad\Longrightarrow\quad
 \Vol(M,g)\geq\Vol(M,g_0).
\end{equation*}
Schoen conjectured this volume comparison in all dimensions
(see \cite{Sch}); in dimension three it follows from Perelman's work on
geometrization by Ricci flow. Agol--Storm--Thurston
(see \cite{AST}*{Conjecture~12.2}) proposed the entropy bound in dimension
three as a route to volume comparison without Ricci flow.
Kazaras--Song--Xu recently disproved that entropy conjecture
(see \cite{KSX}*{Theorem~0.2}).

\begin{theorem}
\label{hd:draw-intro}
Let $(M^3,g_0)$ be a closed hyperbolic three-manifold.
\begin{itemize}
\item There are smooth metrics $g$ arbitrarily close to $g_0$ in the
 tensor $C^0$ topology such that $\Sc_g\geq-6$ and $h(g)>2$.
\item For every $A>0$, there is a smooth metric $g$ such that
 $\Sc_g\geq-6$ and $h(g)>A$.
\end{itemize}
\end{theorem}

Their construction modifies the metric in shrinking tubes around a
simple closed geodesic and then rescales it to restore the scalar
curvature bound. The resulting counterexamples can be $C^0$-close to
$g_0$, but the construction gives no uniform $C^2$ control.
There are also drawstring constructions along codimension-two
submanifolds in higher dimensions (see \cite{KX}).
Kazaras--Song--Xu asked whether the three-dimensional entropy bound
holds for metrics $C^\infty$-close to $g_0$:

\begin{question}\label{q:KSX-local-entropy}
Let $(M^3,g_0)$ be a closed hyperbolic three-manifold. Does there exist
a $C^\infty$ neighborhood $\mathcal U$ of $g_0$ in the space of smooth
Riemannian metrics on $M$ such that
\begin{equation*}
 g\in\mathcal U,\qquad \Sc_g\geq-6
 \quad\Longrightarrow\quad h(g)\leq2?
\end{equation*}
\end{question}
We study volume entropy near $g_0$ under the constraint
$\Sc_g=-n(n-1)$. We compute its Hessian and determine its sign in
every dimension $n\geq3$. This gives an affirmative answer to
Question~\ref{q:KSX-local-entropy}, whereas the analogous local bound
can fail in every dimension $n\geq4$.

Write
\[
 \begin{gathered}
 \Met_{-n(n-1)}=\{g:\Sc_g\equiv-n(n-1)\},\\
 \TT=\{S\in C^\infty(M;S^2T^*M):\tr S=0,\ \delta S=0\}.
 \end{gathered}
\]
Here $g$ ranges over smooth Riemannian metrics, and $\TT$ is the space
of transverse-traceless (TT) tensors. Unless otherwise specified,
operators and inner products are taken with respect to $g_0$.
Our conventions are $\delta S=-\nabla^iS_{ij}$ and
$\Delta=\nabla^*\nabla\geq0$; set $V_0=\Vol(M,g_0)$.
Modulo infinitesimal diffeomorphisms, the tangent space of
$\Met_{-n(n-1)}$ at $g_0$ is $\TT$.
The restricted entropy has a critical point at $g_0$, and its Hessian is
\begin{equation*}
 \begin{split}
 \cH(S,S)&=\frac{n-1}{8}\Cov_0((PS)^\vee,S^\vee)
       -\frac{1}{4nV_0}\ip{PS}{S}_{L^2(M)},\\
 P&=\nabla^*\nabla-2,\qquad S^\vee(x,v)=S_x(v,v).
 \end{split}
\end{equation*}
Here $\Cov_0$ is covariance along the hyperbolic geodesic flow with
respect to the Liouville probability measure $m_0$ on $SM=S_{g_0}M$.
We derive this formula from
Flaminio's fixed-volume entropy variation (see \cite{Flaminio}).
This formula holds in every dimension $n\geq3$.

Covariance is defined for negatively curved geodesic flows, but in
general it is difficult to evaluate explicitly
(see \cite{Gui}, \cite{GKL}, and \cite{Hum}). In the real hyperbolic case, we obtain
the following formula.

\begin{theorem}\label{thm:intro-covariance}
If $0\ne S\in\TT$ satisfies $\Delta S=\lambda S$, then
\[
 \Cov_0(S^\vee,S^\vee)=\kappa_n(\lambda)
       \|S^\vee\|_{L^2(SM,m_0)}^2,
\]
where
\begin{equation*}
 \kappa_n(\lambda)=\frac{J(a+\mu)J(a-\mu)}{J(n+3)},
 \qquad a=\frac{n+3}{2},
\end{equation*}
and
\[
 J(z)=\sqrt\pi\,\frac{\Gamma(z/2)}{\Gamma((z+1)/2)},\qquad
 \mu^2=\frac{(n-1)^2+8}{4}-\lambda.
\]
We choose $\mu\geq0$ when $\mu^2\geq0$, and $\mu$ with positive
imaginary part otherwise; the formula is unchanged by $\mu\mapsto-\mu$.
\end{theorem}

Substitution into the Hessian formula gives the following criterion.
\begin{theorem}\label{hd:main}
For every $n\geq3$, the constrained Hessian at $g_0$ is diagonal in an
$L^2$-orthogonal TT eigenbasis of $\Delta$. If $0\ne S\in\TT$ and $\Delta S=\lambda S$,
then $\lambda\geq n$ and
\begin{equation}\label{hd:main-mode}
 \cH(S,S)=\frac{(\lambda-2)(n-1)}{8}
 \bigl(\kappa_n(\lambda)-\frac{n+2}{n-1}\bigr)
 \|S^\vee\|_{L^2(SM,m_0)}^2.
\end{equation}
The function $\kappa_n$ is strictly decreasing on
$[n,\infty)$ and tends to zero.

\begin{itemize}
\item For $n=3$, $\cH$ is negative definite on $\TT$.
\item For every $n\geq4$, there is exactly one $\lambda_*(n)>n$ satisfying
\begin{equation*}
 \kappa_n(\lambda_*(n))=\frac{n+2}{n-1}.
\end{equation*}
The positive space of $\cH$ is the sum of the TT eigenspaces with
$n\leq\lambda<\lambda_*(n)$; its kernel is the eigenspace at
$\lambda_*(n)$; and its negative space is the sum of those above it.
The positive space and kernel are finite-dimensional, and the negative
space is infinite-dimensional.
\end{itemize}
\end{theorem}

In dimension three, the negative Hessian satisfies a uniform $H^1$
bound. Controlling the nonlinear remainder in one neighborhood of
$g_0$ and then applying conformal normalization gives the local
entropy bound and rigidity.
\begin{corollary}
\label{cor:local-entropy-curves}
Let $(M^3,g_0)$ be a closed hyperbolic three-manifold. There is a
$C^\infty$ neighborhood $\mathcal U$ of $g_0$ such that every smooth
metric $g\in\mathcal U$ satisfies
\[
 \Sc_g\geq-6\quad\Longrightarrow\quad h(g)\leq2.
\]
Equality holds if and only if $g=\varphi^*g_0$ for a diffeomorphism
$\varphi$ of $M$.
\end{corollary}

In dimensions $n\geq4$, a positive direction exists precisely when
the TT spectrum meets $[n,\lambda_*(n))$. A symmetric tensor $S$ is
\emph{Codazzi} if $\nabla_iS_{jk}=\nabla_jS_{ik}$. Trace-free
Codazzi tensors are exactly the TT eigentensors with $\lambda=n$
(Proposition~\ref{hd:bochner}), so every nonzero such tensor gives a
positive direction.

\Needspace{8\baselineskip}
\begin{corollary}
\label{hd:saddle-intro}
If $n\geq4$ and $(M,g_0)$ admits a nonzero trace-free Codazzi tensor,
then $g_0$ is a saddle to second order on
$\mathcal M_{-n(n-1)}$ modulo diffeomorphisms. There are smooth curves
$g_s^+,g_s^-$ in this space through $g_0$ for which, for sufficiently
small nonzero $s$,
\[
 h(g_s^+)>n-1,\qquad h(g_s^-)<n-1.
\]
Such closed hyperbolic manifolds exist in every dimension $n\geq4$.
The curves converge smoothly to $g_0$ as $s\to0$.
\end{corollary}

These manifolds are obtained by the bending construction of
Johnson--Millson (see \cite{JM}); the passage from its cohomology classes to
Codazzi tensors is proved in Section~\ref{hd:codazzi-section}.

Section~\ref{sec:variation} recalls
the entropy variation formulas. Section~\ref{sec:constraint} develops
the scalar-curvature constraint and the TT spectral lower bound, and
Section~\ref{sec:hessian} derives the constrained Hessian.
Sections~\ref{sec:xray}--\ref{sec:criterion} describe the geodesic
X-ray normal operator and the eigenmode criterion.
Sections~\ref{hd:radial-section}--\ref{hd:multiplier-section} compute
the covariance multiplier and its sign threshold.
Section~\ref{sec:three} proves the local entropy bound in dimension
three, and Section~\ref{hd:codazzi-section} constructs the
higher-dimensional saddles.

\medskip
\noindent\textbf{AI Declaration Statement.}
At an early stage, the authors used ChatGPT to perform certain calculations. The output was verified and rewritten by the authors. 

\section{First and second variations of entropy}\label{sec:variation}
Let $(M^n,g)$ be a closed negatively curved manifold, $n\geq3$, and
let $g_s$ be a smooth curve through $g$ with initial velocity $S$.
Volume entropy equals the topological entropy of the geodesic flow.
The first-variation formula of Katok--Knieper--Weiss
(see \cite{KKW}*{Theorem~3}) gives
\begin{equation*}
 \left.\frac{d}{ds}\right|_{s=0}h(g_s)
 =-\frac{h(g)}2\int_{S_gM}S_x(v,v)\,d\mu_g(x,v),
\end{equation*}
where $\mu_g$ is the Bowen--Margulis probability measure, the unique
invariant probability measure of maximal entropy.
Entropy depends smoothly on smooth metric families
(see \cite{KKPW}); the Banach-space regularity needed below is
specified in Lemma~\ref{lem:entropy-remainder}.

From now on, $g_0$ is hyperbolic. Let $SM=S_{g_0}M$, let $\phi_t$
be its geodesic flow, and let $m_0$ be its Liouville probability
measure. At $g_0$, the Bowen--Margulis measure equals $m_0$.
For a symmetric two-tensor $A$, write $A^\vee(x,v)=A_x(v,v)$.
Fiber integration and the volume variation give
\begin{equation*}
 \begin{aligned}
 \int_{SM}A^\vee\,dm_0
 &=\frac{1}{nV_0}\int_M\tr A\,dV_{g_0},\\
 Dh_{g_0}(S)
 &=-\frac{n-1}{2nV_0}\int_M\tr S\,dV_{g_0}
 =-\frac{n-1}{nV_0}D\Vol_{g_0}(S).
 \end{aligned}
\end{equation*}
Thus $g_0$ is critical for entropy at fixed volume.

For a function $f$ on $SM$, write $\overline f=\int_{SM}f\,dm_0$.  For
two H\"older functions $f_1,f_2$, define the centered covariance
\begin{equation}\label{eq:covariance-definition}
 \begin{split}
 \Cov_0(f_1,f_2)
 :=\lim_{T\to\infty}\frac1T\int_{SM}
 &\left(\int_0^T(f_1-\overline f_1)\circ\phi_t\,dt\right)\\
 &\times\left(\int_0^T(f_2-\overline f_2)\circ\phi_t\,dt\right)dm_0.
 \end{split}
\end{equation}
The hyperbolic geodesic flow is exponentially mixing with respect
to $m_0$ (see \cite{Liverani}*{Corollary~2.5}), so covariance is also
given by the absolutely convergent Green--Kubo integral
\begin{equation}\label{eq:cov2}
 \Cov_0(f_1,f_2)
 =\int_{-\infty}^{\infty}\int_{SM}
 (f_1-\overline f_1)(v)(f_2-\overline f_2)(\phi_t v)
 \,dm_0(v)\,dt.
\end{equation}
Invariance of $m_0$ rewrites the expression before the limit
in \eqref{eq:covariance-definition} as
\[
 \int_{-T}^T\left(1-\frac{|t|}{T}\right)
 \int_{SM}(f_1-\overline f_1)(v)
 (f_2-\overline f_2)(\phi_t v)\,dm_0(v)\,dt.
\]
Exponential decay permits passage to the limit by dominated
convergence. Covariance is symmetric, bilinear, and continuous in
any fixed positive H\"older norm. We write
$\Var_0(f)=\Cov_0(f,f)\geq0$.
The formal adjoint of $\delta$ on symmetric tensors is
$(\delta^*\alpha)_{ij}=\tfrac12(\nabla_i\alpha_j+\nabla_j\alpha_i)$.

Flaminio's formula (see \cite{Flaminio}*{Proposition~5.1.1}) gives the
fixed-volume Hessian below (see also \cite{Pollicott}).

\begin{proposition}\label{prop:flaminio}
Let $\widehat g_s$ be a $C^2$ curve of $C^5$ negatively curved metrics
such that $\Vol(M,\widehat g_s)=V_0$, $\widehat g_0=g_0$, and
$\dot{\widehat g}_0=S$.  Then
\begin{equation*}
 \left.\frac{d^2}{ds^2}\right|_0 h(\widehat g_s)
 =\Cov_0\left(T(S)^\vee,\frac{n-1}{2}S^\vee\right),
\end{equation*}
where
\[
T(S)=\frac14 \nabla^*\nabla S-\frac12\delta^*\delta S+\frac12(\tr_{g_0}S)g_0-\frac12 S.
\]
In particular, $T(S)=\tfrac14PS$ for $S\in\TT$, where
$P=\Delta-2$.
\end{proposition}

\section{The scalar-curvature constraint}\label{sec:constraint}

We use the scalar-curvature variation formulas
(see \cite{Besse}*{Section~1.K}) and the TT identities below to compute
the constrained Hessian.

The linearization of Ricci curvature is
\begin{equation*}
 D\Ric_{g}(S)
 =\frac12\nabla^*\nabla S-\delta^*\delta S-\frac12 \Hess_g(\tr_g S)
 -R^{\circ}(S)+\Ric_g\circ S,
\end{equation*}
where, in an orthonormal frame,
\begin{equation*}
    R^{\circ}(S)_{ij}=R_{ikjl}S_{kl},\qquad
    (\Ric_g\circ S)_{ij}=\frac12 (R_{ik}S_{kj}+R_{jk}S_{ki}).
\end{equation*}
Taking the trace and differentiating the inverse metric gives the
scalar-curvature linearization
\begin{equation}\label{eq:scalar-linearization}
 D\Sc_g(H)=\Delta_g(\tr_g H)+\delta_g\delta_g H
            -\langle\Ric(g),H\rangle_g.
\end{equation}
At $g_0$ the formulas simplify as follows.
\begin{equation}\label{eq:ricci-variation-general}
 D\Ric_{g_0}(S)
 =\frac12\nabla^*\nabla S-nS+(\tr S)g_0
 -\delta^*\delta S-\frac12\nabla d(\tr S).
\end{equation}
\begin{equation}\label{eq:scalar-linearization-hyperbolic}
 D\Sc_{g_0}(S)=\nabla^*\nabla(\tr S)+\delta\delta S
 +(n-1)\tr S.
\end{equation}

Every tensor $H$ satisfying $D\Sc_{g_0}(H)=0$ can be put in TT gauge.
To obtain this gauge, solve $\delta\mathcal L_Xg_0=\delta H$. This is an elliptic
self-adjoint equation on vector fields, with
\[
 \int_M\ip{\delta\mathcal L_Xg_0}{X^\flat}\,dV_{g_0}
 =\frac12\int_M|\mathcal L_Xg_0|^2\,dV_{g_0}.
\]
Its kernel consists of Killing fields. When the Ricci curvature is negative
this kernel is zero and therefore the equation is solvable.
Consequently $H-\mathcal L_Xg_0$ is divergence-free. Scalar curvature
is diffeomorphism invariant and constant at $g_0$, so its trace satisfies
$(\Delta+n-1)\tr(H-\mathcal L_Xg_0)=0$, and hence vanishes.
Conversely, every TT tensor lies in the kernel of
$D\Sc_{g_0}$. The gauge is realized along a constrained curve by
pullback under the flow of $-X$, which preserves volume and entropy.

Each TT tensor can be integrated to a constrained curve.

\begin{lemma}\label{lem:integrability}
For every $S\in\TT$ there is a smooth curve
$g_s\in\Met_{-n(n-1)}$ with $g_{s=0}=g_0$ and
$\left.\frac{d}{ds}\right|_{s=0}g_s=S$.
For sufficiently small $s$, the metrics $g_s$ remain negatively curved.
\end{lemma}

\begin{proof}
Fix $k\geq4$ and $0<\alpha<1$. On the open neighborhood of $(0,0)$
where $g_0+sS+ug_0$ is positive definite, consider
\[
 \Phi(s,u):=\Sc(g_0+sS+u g_0)+n(n-1),
 \qquad
 \Phi:\mathbb R\times C^{k,\alpha}(M)
       \longrightarrow C^{k-2,\alpha}(M).
\]
For $h=fg_0$ one has $\tr h=nf$ and $\delta h=-df$.
Equation \eqref{eq:scalar-linearization-hyperbolic} therefore gives
\begin{equation*}
 D_u\Phi(0,0)f
 =D\Sc_{g_0}(fg_0)
 =(n-1)(\nabla^*\nabla+n)f.
\end{equation*}
The operator on the right is an isomorphism
$C^{k,\alpha}\to C^{k-2,\alpha}$ because it is elliptic, self-adjoint,
and strictly positive.  The implicit function theorem produces a curve
$u(s)$ with $u(0)=0$ and $\Phi(s,u(s))=0$.  Since $S$ is TT,
\eqref{eq:scalar-linearization-hyperbolic} gives
$D_s\Phi(0,0)=0$; differentiating $\Phi(s,u(s))=0$ then shows
$\dot u(0)=0$.  Thus
\[
 g_s=g_0+sS+u(s)g_0
\]
has $\dot g_0=S$ and the required scalar curvature.  Elliptic regularity
and the smooth implicit function theorem give a smooth curve when $S$ is
smooth.  Finally, negative sectional curvature is open in the $C^2$
topology.
\end{proof}

Let $g_s\in\Met_{-n(n-1)}$ be a $C^2$ curve through $g_0$,
with its initial velocity in TT gauge, and put
\begin{equation*}
 S:=\dot g_0\in\TT,
 \qquad A:=\ddot g_0,
 \qquad V(s):=\Vol(M,g_s).
\end{equation*}
The scalar constraint determines the integral of $\tr A$.
Applying \eqref{eq:scalar-linearization} at time $s$ and using
$\partial_s\Sc(g_s)=0$ gives
\[
 0=\Delta_{g_s}(\tr_{g_s}\dot g_s)
   +\delta_{g_s}\delta_{g_s}\dot g_s
   -\langle\Ric(g_s),\dot g_s\rangle_{g_s}.
\]
The first two terms integrate to zero on the closed manifold.  Hence, for
every $s$,
\begin{equation*}
 I(s):=\int_M\langle\dot g_s,\Ric(g_s)\rangle_{g_s}\,dV_{g_s}=0.
\end{equation*}

To differentiate $I(s)$, write
\[
 \langle\dot g_s,\Ric(g_s)\rangle_{g_s}
 =g_s^{ia}g_s^{jb}(\dot g_s)_{ij}\Ric(g_s)_{ab},
 \qquad
 \left.\partial_s\right|_0g_s^{ij}=-S^{ij}.
\]
The volume form and the trace satisfy
\begin{equation}\label{eq:basic-volume-variations}
 (dV_{g_s})'\big|_0=\frac12(\tr S)dV_{g_0},
 \qquad
 \left.\partial_s\right|_0\tr_{g_s}\dot g_s
 =\tr A-|S|^2.
\end{equation}
Since $\Ric(g_0)=-(n-1)g_0$ and $\tr S=0$, the volume-form term in
$I'(0)$ vanishes.  The two inverse metrics each contribute
$(n-1)|S|^2$.  Thus
\begin{equation}\label{eq:I-prime-expanded}
 0=I'(0)=\int_M\left\{-(n-1)\tr A
 +\langle D\Ric_{g_0}(S),S\rangle
 +2(n-1)|S|^2\right\}dV_{g_0}.
\end{equation}
The general formula \eqref{eq:ricci-variation-general} reduces, because
$S$ is TT, to
\begin{equation}\label{eq:ricci-variation}
 D\Ric_{g_0}(S)=\frac12\nabla^*\nabla S-nS.
\end{equation}
Substitution into \eqref{eq:I-prime-expanded} yields
\begin{equation}\label{eq:trace-A}
 (n-1)\int_M\tr A\,dV_{g_0}
 =\int_M\left\{
 \frac12\langle\nabla^*\nabla S,S\rangle
 +(n-2)|S|^2\right\}dV_{g_0}.
\end{equation}

The first variation of volume is
\[
 V'(s)=\frac12\int_M\tr_{g_s}\dot g_s\,dV_{g_s}.
\]
Differentiating and using \eqref{eq:basic-volume-variations} gives the
general second-variation identity
\begin{equation}\label{eq:volume-second-general}
 V''(0)=\int_M\left\{
 \frac12\tr A-\frac12|S|^2+\frac14(\tr S)^2
 \right\}dV_{g_0}.
\end{equation}
Here $\tr S=0$, so inserting \eqref{eq:trace-A} into
\eqref{eq:volume-second-general} gives
\begin{equation}\label{eq:volume-second}
 \boxed{
 V''(0)=\frac{1}{4(n-1)}\int_M
 \langle(\nabla^*\nabla-2)S,S\rangle\,dV_{g_0}.}
\end{equation}
Thus the constrained volume variation depends only on $S$.

The sign of the volume variation follows from the TT spectral lower bound.
\begin{proposition}\label{hd:bochner}
For $S\in\TT$, put $C_{ijk}=\nabla_iS_{jk}-\nabla_jS_{ik}$. Then
\begin{equation}\label{hd:codazzi-energy}
 \frac12\|C\|_{L^2}^2=\ip{(\nabla^*\nabla-n)S}{S}_{L^2}.
\end{equation}
In particular $\nabla^*\nabla\geq n$ on $\TT$, with equality exactly for
trace-free Codazzi tensors, that is, tensors satisfying $C=0$.
\end{proposition}
\begin{proof}
The bundle-valued Bochner identity (see \cite{Besse}*{p.~355}) is
\[
 (d_\nabla d_\nabla^*+d_\nabla^*d_\nabla)S
 =\nabla^*\nabla S-R^\circ(S)+\Ric_{g_0}\circ S.
\]
Here $S$ is viewed as a $T^*M$-valued one-form. For trace-free
$S$ at curvature $-1$, the right-hand side is $\Delta S-nS$.
For $S\in\TT$, we also have $d_\nabla^*S=0$ and $d_\nabla S=C$.
With $|C|^2=\sum_{i,j,k}C_{ijk}^2$ in an orthonormal frame, the
bundle-valued two-form norm is
$\|d_\nabla S\|_{L^2}^2=\tfrac12\|C\|_{L^2}^2$,
so integration gives \eqref{hd:codazzi-energy}.
If $S$ is Codazzi, the same identity gives $\Delta S=nS$.
Conversely, equality in \eqref{hd:codazzi-energy} forces $C=0$.
Finally, contracting the Codazzi equation gives
$\nabla^iS_{ij}=\nabla_j\tr S$, so a trace-free Codazzi tensor is TT.
\end{proof}

Combining \eqref{eq:volume-second} and \eqref{hd:codazzi-energy} yields
\[
 V''(0)=\frac{1}{4(n-1)}
 \left(\frac12\|C\|_{L^2}^2+(n-2)\|S\|_{L^2}^2\right)>0
 \qquad(S\ne0).
\]
Thus every nonzero TT direction increases volume to second order
on the scalar-curvature constraint.

\section{The constrained entropy Hessian}\label{sec:hessian}

We combine the fixed-volume entropy formula from
Section~\ref{sec:variation} with the
volume variation forced by the scalar-curvature constraint.
Let $g_s\in\Met_{-n(n-1)}$ be a $C^2$ curve through $g_0$ with
$\dot g_0=S\in\TT$, and put $V(s)=\Vol(M,g_s)$.
Normalize it by setting
\begin{equation*}
 \widehat g_s=c(s)g_s,
 \qquad
 c(s)=\left(\frac{V_0}{V(s)}\right)^{2/n},
 \qquad
 F(s):=c(s)^{1/2}=\left(\frac{V_0}{V(s)}\right)^{1/n}.
\end{equation*}
Then $\Vol(M,\widehat g_s)=V_0$.  Since $S$ is trace-free,
$V'(0)=0$; hence
\begin{equation}\label{eq:normalization-derivatives}
 F(0)=1,
 \qquad F'(0)=0,
 \qquad F''(0)=-\frac{V''(0)}{nV_0},
 \qquad \dot{\widehat g}_0=S.
\end{equation}
For small $s$, the normalized metrics are still negatively curved.

Entropy satisfies $h(cg)=c^{-1/2}h(g)$.  Therefore
\begin{equation}\label{eq:entropy-scaling}
 h(g_s)=F(s)h(\widehat g_s).
\end{equation}
Thus,
\begin{equation*}
\frac{d}{ds}h\left( g_{s}\right) =F'(s)h(\widehat g_s)
+F(s)\frac{d}{ds}h(\widehat g_s).
\end{equation*}
The Katok--Knieper--Weiss formula and the fixed-volume condition give
$\frac{d}{ds}|_0h(g_s)=0$. The TT gauge and diffeomorphism
invariance therefore show that $g_0$ is critical for entropy on
$\Met_{-n(n-1)}$.  Differentiating
\eqref{eq:entropy-scaling} twice and using $h(g_0)=n-1$ and
\eqref{eq:normalization-derivatives}, we obtain
\begin{equation*}
 \left.\frac{d^2}{ds^2}\right|_0h(g_s)
 =\left.\frac{d^2}{ds^2}\right|_0h(\widehat g_s)-\frac{n-1}{nV_0}V''(0).
\end{equation*}
By Proposition~\ref{prop:flaminio} and
\eqref{eq:volume-second} we obtain the following formula.

\begin{theorem}\label{thm:constrained-hessian}
Let $g_s\in\Met_{-n(n-1)}$ be a $C^2$ curve through $g_0$ with
$S=\dot g_0\in\TT$.  Then
\begin{equation}\label{eq:main-hessian}
 \boxed{\begin{aligned}
 \mathcal H(S,S)
 :=\left.\frac{d^2}{ds^2}\right|_0h(g_s)
 &=\frac{n-1}{8}\Cov_0\bigl((PS)^\vee,S^\vee\bigr)
 -\frac{1}{4nV_0}\langle PS,S\rangle_{L^2(M)},\\
 &\hspace{35mm} P=\nabla^*\nabla-2.
 \end{aligned}}
\end{equation}
The right-hand side depends only on $S$ and not on the acceleration of the
chosen constrained curve.  By Lemma~\ref{lem:integrability}, every
$S\in\TT$ occurs in such a curve.
\end{theorem}

\begin{proof}
Proposition~\ref{prop:flaminio} and $T(S)=\frac14PS$ give
\[
 \left.\frac{d^2}{ds^2}\right|_0h(\widehat g_s)
 =\frac{n-1}{8}\Cov_0\bigl((PS)^\vee,S^\vee\bigr).
\]
Multiplying \eqref{eq:volume-second} by
$-(n-1)/(nV_0)$ gives the second term of
\eqref{eq:main-hessian}.  This proves the formula and its independence of the chosen curve.
\end{proof}

\section{The geodesic X-ray normal operator}\label{sec:xray}

Formula \eqref{eq:cov2} identifies covariance with the generalized
X-ray normal operator. Define $\Pi_2$ weakly on symmetric two-tensors whose evaluation functions have zero
$m_0$-mean by
\begin{equation*}
 \langle\Pi_2A,B\rangle_{L^2(M)}
 :=\Cov_0(A^\vee,B^\vee).
\end{equation*}
Microlocal results for Anosov geodesic flows show that, after restriction
to solenoidal tensors, $\Pi_2$ is nonnegative and elliptic of order $-1$
(see \cite{Gui} and \cite{GKL}*{Lemma~2.1}). The convention of adding
the projection onto constants changes this operator only by a
finite-rank smoothing term, which vanishes on TT tensors.  Thus, the
covariance term in \eqref{eq:main-hessian} has order one:
\[
 \Cov_0((PS)^\vee,S^\vee)=\langle\Pi_2PS,S\rangle.
\]
The second term of \eqref{eq:main-hessian} has order two and
has a negative sign.  This determines the sign at high frequencies.

\begin{theorem}\label{thm:negative-directions}
Assume $n\geq3$.  On $\TT$, the quadratic form $\mathcal H$ in
\eqref{eq:main-hessian} is represented by a self-adjoint operator with
negative elliptic principal symbol of order two and domain $H^2\cap\overline{\TT}^{L^2}$.  In particular:
\begin{enumerate}
 \item $\mathcal H$ is not positive definite;
 \item it has infinitely many negative directions;
 \item its positive spectral subspace is finite-dimensional.
\end{enumerate}
Equivalently, $\mathcal H$ is strictly negative on the orthogonal
complement of a sufficiently large finite-dimensional subspace of $\TT$.
\end{theorem}

\begin{proof}
After compression to the TT space, the symmetrization of
$\Pi_2P$ is a pseudodifferential operator of order one.
The orthogonal TT projection has order zero and nonzero principal
symbol, as recalled in Section~\ref{sec:criterion}.  The second term of \eqref{eq:main-hessian} has principal symbol
\[
 -\frac{1}{4nV_0}|\xi|^2\operatorname{Id}
\]
on the TT symbol space.  It is therefore the principal symbol of the full
Hessian.  Standard elliptic spectral theory gives all three conclusions.

On a hyperbolic manifold, $\nabla^*\nabla$ preserves $\TT$.  If
$\nabla^*\nabla S_j=\lambda_jS_j$ with $\lambda_j\to\infty$, the order
$-1$ property gives
\[
 \Var_0(S_j^\vee)
 =O(\lambda_j^{-1/2})\|S_j\|_{L^2(M)}^2.
\]
Substitution in \eqref{eq:main-hessian} shows that
$\mathcal H(S_j,S_j)<0$ for all sufficiently large $j$.
\end{proof}

Thus only finitely many low-frequency modes can contribute positive
directions. We determine these modes in
Sections~\ref{hd:radial-section}--\ref{hd:multiplier-section}.

\section{The eigenmode criterion}\label{sec:criterion}

High-frequency directions are negative in every dimension. For the
remaining modes, integration over the unit tangent spheres gives the
threshold $(n+2)/(n-1)$ below.

The rough Laplacian on symmetric two-tensors is the elliptic operator
\[
 \Delta=\nabla^*\nabla:
 C^\infty(M;S^2T^*M)\longrightarrow C^\infty(M;S^2T^*M).
\]
We consider its restriction to the space $\TT$ of trace-free,
divergence-free tensors.  On our hyperbolic manifold,
\[
 \begin{aligned}
 \tr(\Delta S)&=\Delta(\tr S),\\
 \delta(\Delta S)&=\Delta(\delta S)+(n+1)\delta S+2d(\tr S).
 \end{aligned}
\]
Here each $\Delta$ acts on the indicated tensor bundle.  These
identities show that $\Delta:\TT\to\TT$.

Let $\overline{\mathcal T}_{\mathrm{TT}}^{L^2}$ be the $L^2$ closure of $\TT$.
The resolvent $(1+\Delta)^{-1}$ also preserves this space.
For TT data, taking the trace of the resolvent equation gives
$(1+\Delta)\tr S=0$, hence $\tr S=0$.
Taking the divergence then gives $(\Delta+n+2)\delta S=0$,
hence $\delta S=0$.  The same conclusion holds for data in the
$L^2$ closure by continuity of the resolvent.  Thus the restriction of $\Delta$ is self-adjoint, with domain
\[
 H^2(M;S^2T^*M)\cap\overline{\mathcal T}_{\mathrm{TT}}^{L^2}.
\]
Elliptic regularity and compactness of the Sobolev embedding
$H^2\hookrightarrow L^2$ give compact resolvent and smooth
eigentensors.

To prove that the TT spectrum is unbounded, we also need that
$\TT$ is infinite-dimensional.  Restrict
$\delta$ to trace-free symmetric tensors and denote its adjoint by
\[
 D_0\beta=\delta^*\beta+\frac1n(\delta\beta)g_0.
\]
The symbol of $\delta$ on trace-free tensors is surjective, so
$\delta D_0$ is elliptic.  Let $Q$ be its inverse on the orthogonal
complement of its kernel, extended by zero on the kernel.
Then $\Pi=\operatorname{Id}-D_0Q\delta$ is the orthogonal
projection onto trace-free divergence-free tensors.  The elliptic
parametrix theorem makes $\Pi$ an operator of order zero whose
principal symbol projects onto
\[
 \{A\in S^2T_x^*M:\tr A=0,\ A(\xi^\sharp,\cdot)=0\},
 \qquad \xi\ne0.
\]
This space has dimension $(n-2)(n+1)/2>0$.
Thus $\Pi$ has nonzero principal symbol and cannot have finite rank;
its smooth range $\TT$ is infinite-dimensional.

For each eigenvalue $\lambda$, let
\[
 E_\lambda=\{S\in\TT:\Delta S=\lambda S\}
\]
be its eigenspace; a nonzero element is called a TT eigentensor.  The
spectral theorem gives finite-dimensional eigenspaces and an
$L^2$-orthonormal basis $\{S_j\}_{j\geq1}$ of
$\overline{\mathcal T}_{\mathrm{TT}}^{L^2}$ consisting of smooth TT eigentensors, with
\[
 \Delta S_j=\lambda_jS_j,\qquad
 n\leq\lambda_1\leq\lambda_2\leq\cdots,\qquad
 \lambda_j\longrightarrow\infty.
\]
The eigenvalues are listed with multiplicity;
Proposition~\ref{hd:bochner} gives $\lambda_j\geq n$.

Every smooth $S\in\TT$ has the spectral expansion
\[
 S=\sum_{j=1}^{\infty}s_jS_j,
 \qquad s_j=\ip{S}{S_j}_{L^2(M)}.
\]
The series converges in every Sobolev norm, which we write as
\[
 \|S\|_{H^r}:=\|(1+\Delta)^{r/2}S\|_{L^2}.
\]

Suppose $S\in\TT$ is a nonzero eigentensor,
\[
 \nabla^*\nabla S=\lambda S.
\]
By Proposition~\ref{hd:bochner}, $\lambda\geq n$, so $\lambda-2>0$.  Formula
\eqref{eq:main-hessian} becomes
\begin{equation*}
 \mathcal H(S,S)
 =\frac{\lambda-2}{8}
 \left\{(n-1)\Var_0(S^\vee)
 -\frac{2}{nV_0}\|S\|_{L^2(M)}^2\right\}.
\end{equation*}
A trace-free tensor has $\overline{S^\vee}=0$. The fourth-moment identity for the rotationally invariant
probability measure $d\sigma$ on a unit sphere is
\[
 \int v_iv_jv_kv_\ell\,d\sigma
 =\frac{\delta_{ij}\delta_{k\ell}+\delta_{ik}\delta_{j\ell}
       +\delta_{i\ell}\delta_{jk}}{n(n+2)}.
\]
Contracting with two trace-free tensors and integrating over $M$ gives
\begin{equation}\label{hd:fiber}
 \int_{SM}A^\vee S^\vee\,dm_0
 =\frac{2}{n(n+2)V_0}\ip{A}{S}_{L^2(M)}
 \quad(\tr A=\tr S=0).
\end{equation}
Therefore, we arrive at
\begin{equation*}
 \mathcal H(S,S)
 =\frac{(\lambda-2)(n-1)}{8}
 \left\{\Var_0(S^\vee)
 -\frac{n+2}{n-1}\|S^\vee\|_{L^2(SM,m_0)}^2\right\}.
\end{equation*}

Consequently,
\begin{equation}\label{eq:variance-criterion-general}
 \boxed{
 \mathcal H(S,S)>0
 \quad\Longleftrightarrow\quad
 \frac{\Var_0(S^\vee)}{\|S^\vee\|_{L^2(SM,m_0)}^2}
 >\frac{n+2}{n-1}.}
\end{equation}
For $n=3$, this becomes
\begin{equation*}
 \boxed{
 \mathcal H(S,S)>0
 \quad\Longleftrightarrow\quad
 \Var_0(S^\vee)>\frac52\|S^\vee\|_{L^2(SM,m_0)}^2.}
\end{equation*}

The quotient in \eqref{eq:variance-criterion-general} is variance
divided by the squared $L^2(SM,m_0)$ norm. By \eqref{hd:fiber}, it is
$n(n+2)V_0/2$ times the $L^2(M)$ Rayleigh quotient of $\Pi_2$.  In general dimension, a saddle point can occur to
second order only if the finite-dimensional low-frequency part of
$\mathcal H$ has a positive eigenvalue. In dimension three the
normalized variance is strictly less than $5/2$, as proved below.

\section{Radial reduction of the correlation function}\label{hd:radial-section}

Fix $S\in E_\lambda$ and $x\in M$.  Lift $S$ to $\mathbb H^n$ and
choose a lift of $x$, using the same notation for both.  In geodesic polar coordinates
centered at $x$,
\[
 g_0=dr^2+\sinh^2r\,\sigma,
\]
where $\sigma$ is the round metric on the unit sphere in $T_xM$.
Set $e_0=\partial_r$ and extend a local $\sigma$-orthonormal frame
$\{e_a\}_{a=1}^{n-1}$ independently of $r$.  The angular fields have
length $\sinh r$, and the hyperbolic orthonormal frame is
\[
 e_0,\quad (\sinh r)^{-1}e_1,\quad\ldots,\quad
 (\sinh r)^{-1}e_{n-1}.
\]
For $r>0$, define the radial component $U$ and angular one-form $W$ by
\[
 U=S(e_0,e_0),\qquad W_a=S(e_0,e_a).
\]
The function $U(r,\xi)$ is the evaluation of $S$ on the velocity of
the geodesic from $x$ in direction $\xi$.  We use the nonnegative
spherical Laplacian $\Delta_\sigma=-\operatorname{div}_\sigma\nabla^\sigma$.

\begin{lemma}\label{hd:polar-laplace}
For $r>0$, the radial component $U$ satisfies
\begin{equation}\label{hd:polar-equation}
 \partial_r^2U+(n+3)\coth r\,\partial_rU
 -\frac{\Delta_\sigma U}{\sinh^2r}
 +(\lambda+2n\coth^2r)U=0.
\end{equation}
\end{lemma}

\begin{proof}
In this lifted frame, the connection formulas are
\[
 \begin{aligned}
 \nabla_{e_0}e_0&=0,&
 \nabla_{e_a}e_0=\nabla_{e_0}e_a&=\coth r\,e_a,\\
 \nabla_{e_a}e_b&=\nabla^\sigma_{e_a}e_b
 -\sinh r\cosh r\,\delta_{ab}e_0.
 \end{aligned}
\]
Trace-freeness gives $\sum_aS(e_a,e_a)=-\sinh^2r\,U$.
Applying the connection formulas to $W$ therefore gives
\[
 \sum_a(\nabla_{e_a}S)(e_0,e_a)
 =\operatorname{div}_\sigma W+n\sinh r\cosh r\,U.
\]
Taking the trace in the hyperbolic orthonormal frame, the radial
component of $\delta S=0$ gives
\begin{equation}\label{hd:polar-div}
 \frac{\operatorname{div}_\sigma W}{\sinh^2r}
 =-\partial_rU-n\coth r\,U.
\end{equation}

To compute $(\Delta S)(e_0,e_0)$, first note that
\[
 (\nabla_{e_0}S)(e_0,e_0)=\partial_rU,\qquad
 (\nabla_{e_a}S)(e_0,e_0)=e_aU-2\coth r\,W_a.
\]
Differentiate once more at a point where the spherical frame is normal,
and take the trace in the hyperbolic orthonormal frame.  Each angular
term contributes a factor $(\sinh r)^{-2}$, giving
\[
 \begin{aligned}
 -(\Delta S)(e_0,e_0)
 &=\partial_r^2U+(n-1)\coth r\,\partial_rU
      -\frac{\Delta_\sigma U}{\sinh^2r}\\
 &\quad-\frac{4\coth r}{\sinh^2r}\operatorname{div}_\sigma W
      -2n\coth^2r\,U.
 \end{aligned}
\]
Substituting \eqref{hd:polar-div} and
$(\Delta S)(e_0,e_0)=\lambda U$ proves \eqref{hd:polar-equation}.
\end{proof}

To compute the correlation with a smooth trace-free symmetric
two-tensor $A$, we average $U$ against $A_x(\xi,\xi)$ over initial
directions, then integrate over $x$.  For $\xi\in S_xM$, write
$\gamma_{x,\xi}(t)=\exp_x(t\xi)$ and define
\[
 F_{x,A,S}(t)=\int_{S_xM}A_x(\xi,\xi)
 S_{\gamma_{x,\xi}(t)}
       (\dot\gamma_{x,\xi}(t),\dot\gamma_{x,\xi}(t))\,d\sigma_x(\xi),
\]
where $d\sigma_x$ is the rotationally invariant probability measure
on $S_xM$.
By the definition of the geodesic flow,
\[
 \int_{SM}A^\vee(S^\vee\circ\phi_t)\,dm_0
 =\frac{1}{V_0}\int_M F_{x,A,S}(t)\,dV_{g_0}(x).
\]
The radial equation determines $F_{x,A,S}$ from its value at zero.

\begin{proposition}\label{pro:correlation}
For $S\in E_\lambda$ and every smooth trace-free symmetric two-tensor
$A$, the correlation satisfies
\begin{equation}\label{eq:correlation-factor}
 \int_{SM}A^\vee(S^\vee\circ\phi_t)\,dm_0
 =\varphi_{n,\lambda}(t)\int_{SM}A^\vee S^\vee\,dm_0
 \qquad(t\in\mathbb R),
\end{equation}
where $\varphi_{n,\lambda}\colon\mathbb R\to\mathbb R$ is the unique
smooth even solution of
\begin{equation}\label{hd:correlation-ode}
 \varphi''+(n+3)\coth r\,\varphi'+(\lambda+2n)\varphi=0,
 \qquad \varphi(0)=1,\quad \varphi'(0)=0.
\end{equation}
The differential equation holds for $r>0$.
\end{proposition}

\begin{proof}
Fix $x\in M$ and abbreviate $F(t)=F_{x,A,S}(t)$.
The geodesic flow is smooth, so $F$ is smooth.  Replacing $\xi$ by
$-\xi$ changes $\gamma_{x,\xi}(t)$ to $\gamma_{x,\xi}(-t)$ and
reverses its velocity.  Since both tensor evaluations are quadratic,
this substitution gives $F(-t)=F(t)$ and hence $F'(0)=0$.

Set $Y_{A_x}(\xi)=A_x(\xi,\xi)$.  In the polar coordinates above,
\[
 F(r)=\int_{S_xM}Y_{A_x}(\xi)U(r,\xi)\,d\sigma_x(\xi).
\]
Since $A_x$ is trace-free, $Y_{A_x}$ is a spherical harmonic of
degree two and $\Delta_\sigma Y_{A_x}=2nY_{A_x}$.  Thus integration
by parts on the sphere gives
\[
 \int_{S_xM}Y_{A_x}\Delta_\sigma U\,d\sigma_x
 =\int_{S_xM}(\Delta_\sigma Y_{A_x})U\,d\sigma_x
 =2nF(r).
\]
Averaging \eqref{hd:polar-equation} against $Y_{A_x}$ and
differentiating under the integral in $r$ gives
\[
 F''+(n+3)\coth r\,F'
 +\left(\lambda+2n\coth^2r-\frac{2n}{\sinh^2r}\right)F=0.
\]
Since $\coth^2r-(\sinh r)^{-2}=1$, this is precisely the
differential equation in \eqref{hd:correlation-ode}.

Because $\coth r$ is singular at zero, we prove uniqueness among
smooth even solutions directly.
Put $c=\lambda+2n$ and $p(r)=(\sinh r)^{n+3}$.  A smooth solution
$f$ satisfies $(pf')'=-cpf$ and $p(r)f'(r)\to0$ as $r\downarrow0$.
Integrating twice gives
\[
 f(r)=f(0)-c(Kf)(r),\qquad
 (Kf)(r)=\int_0^r\frac{1}{p(u)}\int_0^u p(v)f(v)\,dv\,du.
\]
Monotonicity of $p$ gives the estimate
\[
 \|Kh\|_{C[0,\varepsilon]}
 \leq \frac{\varepsilon^2}{2}\|h\|_{C[0,\varepsilon]}.
\]
For $|c|\varepsilon^2/2<1$, the map $h\mapsto b-cKh$ is a
contraction.  Thus at most one smooth solution has initial value $b$.
For existence, the regular-singular
equation has indicial roots $0$ and $-(n+2)$.  Its regular branch
has a convergent even power series with arbitrary constant term $b$.  Continuation for the linear equation on $(0,\infty)$
and even reflection give the global solution; taking $b=1$
defines $\varphi_{n,\lambda}$.

By linearity, $F(0)\varphi_{n,\lambda}$ solves the same equation as
$F$ and has the same value at zero.  Uniqueness and evenness thus give
\begin{equation*}
 F_{x,A,S}(t)=\varphi_{n,\lambda}(t)F_{x,A,S}(0)
 \qquad(t\in\mathbb R).
\end{equation*}
Integrating over $x$ with measure
$dV_{g_0}/V_0$ proves \eqref{eq:correlation-factor}.
\end{proof}

We now integrate $\varphi_{n,\lambda}$ to obtain the covariance.

\section{The covariance multiplier}\label{hd:multiplier-section}

\subsection{Solving and integrating the radial equation}

We solve
\[
 \varphi''+(n+3)\coth t\,\varphi'+(\lambda+2n)\varphi=0,
 \qquad \varphi(0)=1,\quad\varphi'(0)=0.
\]
Set $x=\cosh t$. The differential equation becomes
\[
 (x^2-1)\varphi_{xx}+(n+4)x\varphi_x+(\lambda+2n)\varphi=0.
\]
With $z=(1-x)/2=-\sinh^2(t/2)$, this is the hypergeometric equation
\begin{equation*}
 z(1-z)\varphi_{zz}
 +\bigl[c-(q_++q_-+1)z\bigr]\varphi_z-q_+q_-\varphi=0,
\end{equation*}
where
\[
 a=\frac{n+3}{2},\qquad c=\frac{n+4}{2},\qquad
 q_\pm=a\pm\mu,\qquad
 \mu^2=a^2-(\lambda+2n)=\frac{(n-1)^2+8}{4}-\lambda.
\]
Thus the regular solution is
${}_2F_1(q_+,q_-;c; -\sinh^2(t/2))$
(see \cite{AAR}). As in the introduction, $\mu$ is nonnegative real
or has positive imaginary part.

For $\Re z>0$, define
\[
 \begin{aligned}
 J(z)&=\int_{\mathbb R}(\cosh r)^{-z}\,dr
      =\int_{-1}^1(1-u^2)^{z/2-1}\,du\\
     &=\sqrt\pi\,\frac{\Gamma(z/2)}{\Gamma((z+1)/2)}.
 \end{aligned}
\]
The substitution $u=\tanh r$, followed by the beta integral, gives
these identities. The radial solution has the representation
\begin{equation}\label{hd:spherical-integral}
 \varphi_{n,\lambda}(t)=\frac{1}{J(2a)}
 \int_{-1}^1(\cosh t-u\sinh t)^{-q_+}(1-u^2)^{a-1}\,du.
\end{equation}
The right-hand side is smooth and even in $t$, and its value at zero
is $1$.
The base of the complex power is positive, so we use its real
logarithm. Put $D=\cosh t-u\sinh t$ and $w=(1-u^2)^{a-1}$.
Direct differentiation gives, for $t\ne0$,
\[
 w\bigl(\partial_t^2+2a\coth t\,\partial_t+q_+(2a-q_+)\bigr)D^{-q_+}
 =-\frac{q_+}{\sinh t}\partial_u\bigl((1-u^2)^aD^{-q_+-1}\bigr).
\]
The boundary term at $u=\pm1$ vanishes. Since $2a=n+3$ and
$q_+(2a-q_+)=\lambda+2n$, integration proves the radial equation.
Uniqueness from Proposition~\ref{pro:correlation} then gives
\eqref{hd:spherical-integral}.

\begin{theorem}\label{thm:covariance-multiplier}
Let $0\ne S\in\TT$ satisfy $\Delta S=\lambda S$, and let $A$ be any
smooth trace-free symmetric two-tensor. Then
\begin{equation}\label{hd:covariance-final}
 \Cov_0(A^\vee,S^\vee)=\kappa_n(\lambda)
       \int_{SM}A^\vee S^\vee\,dm_0,
 \qquad
 \kappa_n(\lambda)=\frac{J(a+\mu)J(a-\mu)}{J(n+3)}.
\end{equation}
In particular, this proves Theorem~\ref{thm:intro-covariance}.
\end{theorem}

\begin{proof}
Substituting $u=\tanh r$ in \eqref{hd:spherical-integral} gives
\[
 \varphi_{n,\lambda}(t)=\frac{1}{J(2a)}
 \int_{\mathbb R}(\cosh r)^{-q_-}
                    (\cosh(r-t))^{-q_+}\,dr.
\]
We first justify absolute convergence of the time integral.
If $\mu$ is purely imaginary, $\Re q_\pm=a>0$.
If $\mu$ is real, $\lambda\geq n$ implies
$0\leq\mu\leq(n-3)/2=a-3$, so $q_\pm\geq3$.
Consequently, the substitution $s=r-t$ gives
\[
 \int_{\mathbb R^2}
 \left|(\cosh r)^{-q_-}(\cosh(r-t))^{-q_+}\right|\,dr\,dt
 =J(\Re q_-)J(\Re q_+)<\infty.
\]
Thus $\varphi_{n,\lambda}\in L^1(\mathbb R)$, and Fubini's theorem
gives
\begin{equation*}
 \int_{\mathbb R}\varphi_{n,\lambda}(t)\,dt
 =\frac{J(q_-)J(q_+)}{J(2a)}=\kappa_n(\lambda).
\end{equation*}
Both evaluation functions have zero mean. Combining
\eqref{eq:cov2} with Proposition~\ref{pro:correlation} yields
\[
 \begin{aligned}
 \Cov_0(A^\vee,S^\vee)
 &=\int_{\mathbb R}\int_{SM}A^\vee(S^\vee\circ\phi_t)\,dm_0\,dt\\
 &=\left(\int_{\mathbb R}\varphi_{n,\lambda}(t)\,dt\right)
       \int_{SM}A^\vee S^\vee\,dm_0,
 \end{aligned}
\]
which proves the formula.
\end{proof}

For real $\mu$, both arguments of $J$ are positive, and for purely
imaginary $\mu$ they are complex conjugates. Since the Gamma
function has no zeros and no poles in the right half-plane,
$\kappa_n(\lambda)>0$ in both cases.
Equations~\eqref{hd:fiber} and \eqref{hd:covariance-final} also show
that $L^2(M)$-orthogonal TT eigentensors have zero covariance,
including those in the same eigenspace.

\subsection{Strict monotonicity}

Set $z=\mu^2=a^2-(\lambda+2n)$ and
$\lambda_c=2+(n-1)^2/4$. The Gamma product gives
\[
 \frac{\kappa_n(\lambda)}{\kappa_n(\lambda_c)}
 =\frac{J(a+\mu)J(a-\mu)}{J(a)^2}.
\]
The Weierstrass product for the Gamma function therefore yields
\begin{equation*}
 \frac{\kappa_n(\lambda)}{\kappa_n(\lambda_c)}
 =\prod_{k=0}^\infty
   \frac{1-z/(a+1+2k)^2}{1-z/(a+2k)^2}.
\end{equation*}
For $\lambda\geq n$ we have
$z\leq(a-3)^2<a^2$, so all factors are positive.
On compact subintervals of $z<a^2$, the logarithms of the factors
and their derivatives are uniformly $O(k^{-3})$. Thus their series
converge locally uniformly and may be differentiated term by term:
\begin{equation*}
 \frac{d}{dz}\log\kappa_n
 =\sum_{k=0}^\infty\left(
   \frac{1}{(a+2k)^2-z}-\frac{1}{(a+1+2k)^2-z}\right)>0.
\end{equation*}
Since $dz/d\lambda=-1$, the multiplier is strictly decreasing on
the entire interval $[n,\infty)$.

For large $\lambda$, write $\mu=i\tau$ with
$\tau=\sqrt{\lambda-\lambda_c}$. Stirling's formula gives
$|J(a+i\tau)|^2=(2\pi/\tau)(1+O(\tau^{-2}))$, and hence
\begin{equation}\label{hd:high-frequency}
 \kappa_n(\lambda)=\frac{2\pi}{J(n+3)\sqrt\lambda}
                    \bigl(1+O(\lambda^{-1})\bigr)
 \qquad(\lambda\to\infty).
\end{equation}

\subsection{The change of sign at dimension four}

At $\lambda=n$, the arguments $a+\mu$ and $a-\mu$ are $n$ and $3$.
Thus
\begin{equation}\label{hd:bottom-value}
 K_n:=\kappa_n(n)=\frac{J(3)J(n)}{J(n+3)},\qquad J(3)=\frac\pi2.
\end{equation}
In dimension three, $K_3=15\pi^2/64<5/2$. Strict monotonicity
makes \eqref{hd:main-mode} negative on every nonzero TT eigenmode.
The Gamma recurrence and reflection formulas also give
the elementary expression
\[
 \kappa_3(3+\tau^2)=\frac{15\pi}{8\tau}
       \frac{1+\tau^2}{4+\tau^2}\tanh\frac{\pi\tau}{2},
 \qquad\tau\geq0,
\]
with value $15\pi^2/64$ at $\tau=0$ by continuity.

For $n\geq4$, set $R_n=(n-1)K_n/(n+2)$.
The recurrence $J(z+2)=zJ(z)/(z+1)$ gives
\begin{equation}\label{hd:bottom-recurrence}
 \begin{gathered}
 R_4=\frac{16}{15}>1,\qquad R_5=\frac{15\pi^2}{128}>1,\\
 \frac{R_{n+2}}{R_n}=\frac{n(n+2)}{(n-1)(n+3)}>1.
 \end{gathered}
\end{equation}
Both parity sequences therefore satisfy $R_n>1$.
Strict monotonicity and \eqref{hd:high-frequency} give a unique
$\lambda_*(n)>n$ at which $\kappa_n(\lambda_*(n))=(n+2)/(n-1)$.

\begin{proof}[Completion of the proof of Theorem~\ref{hd:main}]
Insert \eqref{hd:covariance-final} and \eqref{hd:fiber} into
\eqref{eq:main-hessian} to obtain \eqref{hd:main-mode}.
The preceding comparisons determine the sign on every eigenspace.
By Section~\ref{sec:criterion}, there is an $L^2(M)$-orthonormal
TT eigenbasis $\{S_j\}$ with eigenvalues $\lambda_j\to\infty$.
Orthogonal eigenmodes have zero cross-covariance, so the Hessian is
diagonal in this basis.

For a smooth TT tensor $S=\sum_js_jS_j$, its partial sums converge
in every Sobolev norm, as do their images under $P$.
Continuity of covariance in a fixed H\"older norm permits passage
to the limit in \eqref{eq:main-hessian}. The resulting formula is
\begin{equation}\label{eq:full-spectral-hessian}
 \cH(S,S)=\frac{1}{4n(n+2)V_0}
 \sum_j(\lambda_j-2)
 \bigl((n-1)\kappa_n(\lambda_j)-(n+2)\bigr)s_j^2.
\end{equation}
By \eqref{hd:high-frequency}, the coefficient of $s_j^2$ is
$-\lambda_j/(4nV_0)+O(\sqrt{\lambda_j})$.
In particular, its absolute value is bounded by $C(1+\lambda_j)$,
so the series converges absolutely and extends the Hessian
continuously to $H^1\cap\overline{\TT}^{L^2}$.
In dimension three every coefficient is negative. In higher
dimensions, only finitely many eigenvalues lie at or below
$\lambda_*(n)$, while infinitely many lie above it. This proves
all the assertions.
\end{proof}

The values $n$ and $\lambda_*(n)$ need not be eigenvalues on a given
manifold. The threshold determines the sign of each eigenspace
that is present.

\section{The local entropy bound in dimension three}\label{sec:three}

Throughout this section, $n=3$. We derive a uniform $H^1$ bound for
the Hessian and estimate the nonlinear remainder near $g_0$.

\begin{proposition}
\label{thm:negative-definite-3d}\label{cor:coercive-3d}
Let
\[
 \eta=5-\frac{15\pi^2}{32}>0,\qquad c_0=\frac{\eta}{240V_0}.
\]
Then
\begin{equation}\label{eq:negative-definite-3d}
 \mathcal H(S,S)\leq-c_0\|S\|_{H^1}^2\qquad(S\in\TT).
\end{equation}
\end{proposition}
\begin{proof}
For an eigentensor with eigenvalue $\lambda\geq3$,
\eqref{hd:main-mode}, \eqref{hd:fiber}, and
$\kappa_3(\lambda)\leq K_3=15\pi^2/64$ give
\[
 \mathcal H(S,S)
 \leq-\frac{\eta}{60V_0}(\lambda-2)\|S\|_{L^2}^2
 \leq-\frac{\eta}{240V_0}(1+\lambda)\|S\|_{L^2}^2.
\]
The last inequality is $\lambda-2\geq(1+\lambda)/4$ for $\lambda\geq3$.
Sum using \eqref{eq:full-spectral-hessian} to obtain
\eqref{eq:negative-definite-3d}.
\end{proof}

For the local coordinates below, write
\[
 \TT^r=\{S\in H^r(M;S^2T^*M):\tr S=0,\ \delta S=0\},
\]
where the equations are understood weakly. This is a closed Hilbert
subspace of $H^r$, and smooth TT tensors are dense in it by the spectral
decomposition. We use $\TT^8$ as the domain for the implicit function theorem.

\begin{lemma}
\label{lem:constraint-coordinates}
In the affine slice $\delta(g-g_0)=0$, the metrics with scalar curvature
$-6$ near $g_0$ admit a smooth $H^8$ parametrization
\[
 G(S)=g_0+S+\mathcal R w(S),\qquad S\in\TT^8,
 \qquad w(0)=0,\quad Dw(0)=0.
\]
Here $w(S)\in H^6$, and $\mathcal R:H^6\to H^8$ is a fixed bounded
linear operator. Every smooth metric sufficiently close to $g_0$ with
scalar curvature $-6$ is, after pullback by a diffeomorphism, of this
form. Its coordinate $S$ is small in $H^{22}$ when the original metric
is sufficiently close in the smooth topology.
\end{lemma}
\begin{proof}
Let $L=D\Sc_{g_0}$. Its formal adjoint on functions is
\[
 L^*f=\Hess f+(\Delta f)g_0+2fg_0.
\]
Diffeomorphism invariance of scalar curvature gives $\delta L^*f=0$.
Its trace is $2(\Delta+3)f$, and on divergence-free tensors
$LH=(\Delta+2)\tr H$. Thus
\[
 LL^*=2(\Delta+2)(\Delta+3),\qquad
 \mathcal R=L^*\bigl[2(\Delta+2)(\Delta+3)\bigr]^{-1}
\]
satisfies $L\mathcal R=\operatorname{Id}$ and
$\delta\mathcal R=0$. Elliptic estimates make $\mathcal R$ bounded
from $H^{r-2}$ to $H^r$. Every divergence-free tensor $H\in H^r$ decomposes as
\[
 H=S+\mathcal R LH,\qquad S=H-\mathcal R LH\in\TT^r.
\]
The equations $LS=0$ and $\delta S=0$ give
$(\Delta+2)\tr S=0$, hence $\tr S=0$.

The map
\[
 (S,w)\longmapsto\Sc(g_0+S+\mathcal R w)+6,
 \qquad \TT^8\times H^6\longrightarrow H^6,
\]
is smooth on the open set of positive metrics. Its derivative in $w$
at zero is the identity. The Banach implicit function theorem gives
$w(S)$, and differentiation in a TT direction gives $Dw(0)=0$.
The decomposition above and uniqueness in that theorem show that
$G$ parametrizes all nearby constrained metrics in the affine slice.
The constructions at higher Sobolev regularity agree wherever both
are defined, by uniqueness.

The slice is obtained by solving
$\delta(\psi^*g-g_0)=0$ near $\psi=\operatorname{Id}$.
Its linearization in the vector field generating $\psi$ is
$X\mapsto\delta\mathcal L_Xg_0$. This elliptic operator has zero
kernel, as proved in Section~\ref{sec:constraint}. With $g$ in
$C^{25,1/2}$, the equation defines a $C^1$ map from vector fields in
$C^{24,1/2}$ to one-forms in $C^{22,1/2}$, and its linearization is an
isomorphism between these spaces. The implicit function theorem
therefore gives a nearby $\psi$, continuously depending on $g$.
The pullback is close in $C^{23,1/2}$, hence in $H^{22}$.
Since $H\mapsto H-\mathcal R LH$ is bounded on $H^{22}$, its TT
coordinate is also small in $H^{22}$. For smooth metrics, elliptic
regularity for this gauge equation gives a smooth diffeomorphism.
\end{proof}

\begin{lemma}
\label{lem:entropy-remainder}
After shrinking the coordinate neighborhood in
Lemma~\ref{lem:constraint-coordinates}, the metrics $G(S)$ have
negative sectional curvature, and the function $F(S)=h(G(S))$ is $C^3$ near zero in $\TT^8$.
There is a constant $C$ such that
\begin{equation}\label{eq:uniform-remainder}
 \left|F(S)-2-\tfrac12\mathcal H(S,S)\right|
 \leq C\|S\|_{H^1}^2\|S\|_{H^{22}}
\end{equation}
for smooth $S$ sufficiently small in $H^{22}$.
\end{lemma}
\begin{proof}
We prove $C^3$ regularity in the $H^8$ topology.
On the fixed bundle $S_{g_0}M$, the geodesic vector field depends
smoothly on the metric, losing one spatial derivative. Structural
stability gives a positive H\"older time-change function $a_g$ whose
dependence on a $C^4$ vector field is $C^3$ with values in a fixed
H\"older space (see \cite{KKPW}*{Proposition~2.2 and its proof,
pp.~591--592}). Thus $g\mapsto a_g$ is $C^3$ on a sufficiently small
$C^{5,1/2}$ neighborhood. Normalize the time change so that its integral
on a closed $g_0$ orbit is the period of the corresponding $g$ orbit.
The entropy is determined by
\[
 \mathscr P_{g_0}(-h(g)a_g)=0,
\]
where $\mathscr P_{g_0}$ is the pressure of the fixed geodesic flow.
Pressure is analytic on this H\"older space, and its derivative in the
scalar variable $h$ is $-\int a_g\,d\mu$, which is strictly negative;
here $\mu$ is the equilibrium probability measure
(see \cite{KKPW}*{Proposition~2.1 and the proof of Theorem~2,
pp.~589--590 and~594}). The implicit function theorem makes
$g\mapsto h(g)$ a $C^3$ function on this neighborhood.
This is Fr\'echet differentiability, as specified on p.~594 of the cited paper.
All the metrics are negatively curved, so topological and volume
entropy coincide.

In dimension three, $H^8$ embeds continuously into $C^{5,1/2}$.
Composing with $G$ therefore makes $F$ a $C^3$ function on $\TT^8$.
Here $DF(0)=0$ and $D^2F(0)=\mathcal H$, since $G(0)=g_0$ and
$DG(0)S=S$. Continuity of $D^3F$ bounds its operator norm on a sufficiently
small ball centered at zero. Taylor's formula on this ball gives
\[
 |F(S)-2-\tfrac12\mathcal H(S,S)|\leq C\|S\|_{H^8}^3.
\]
Sobolev interpolation, with $8=\tfrac23\cdot1+\tfrac13\cdot22$, gives
\[
 \|S\|_{H^8}^3\leq\|S\|_{H^1}^2\|S\|_{H^{22}},
\]
which proves \eqref{eq:uniform-remainder}.
\end{proof}

\begin{proposition}
\label{prop:constrained-local-maximum}
For all smooth metrics $g$ sufficiently close to $g_0$ with
$\Sc_g=-6$, one has $h(g)\leq2$. Equality holds precisely when
$g$ is a pullback of $g_0$ by a diffeomorphism.
\end{proposition}
\begin{proof}
By Lemma~\ref{lem:constraint-coordinates}, write $\psi^*g=G(S)$ with
$S$ small in $H^{22}$. Entropy is invariant under pullback.
Proposition~\ref{cor:coercive-3d} and
Lemma~\ref{lem:entropy-remainder} imply
\[
 h(g)-2\leq\left(-\frac {c_0}2+C\|S\|_{H^{22}}\right)\|S\|_{H^1}^2
 \leq-\frac {c_0}4\|S\|_{H^1}^2
\]
after shrinking the neighborhood. Equality forces $S=0$, hence
$\psi^*g=G(0)=g_0$. Conversely, pullbacks of $g_0$ have entropy $2$.
\end{proof}

\begin{proposition}
\label{prop:conformal-reduction}
For smooth metrics $g$ sufficiently close to $g_0$, there is a positive
function $u_g$ close to $1$ such that
\[
 \widehat g=u_g^4g,\qquad \Sc_{\widehat g}=-6.
\]
The assignment $g\mapsto\widehat g$ is continuous in the smooth topology
and fixes $g_0$. If $\Sc_g\geq-6$, then
\[
 u_g\leq1,\qquad \widehat g\leq g,\qquad h(g)\leq h(\widehat g).
\]
Consequently, Question~\ref{q:KSX-local-entropy} is equivalent to the
existence of a neighborhood of $g_0$ on which $h\leq2$ for metrics
in $\Met_{-6}$.
\end{proposition}
\begin{proof}
Fix $k\geq4$ and $0<\alpha<1$. The conformal scalar-curvature
formula reduces the normalization to
\[
 F(g,u):=8\Delta_g u+\Sc_g u+6u^5=0,
\]
where $\Delta_g$ is the nonnegative Laplacian. At $(g_0,1)$,
\[
 D_uF(g_0,1)=8(\Delta_{g_0}+3),
\]
an isomorphism from $C^{k,\alpha}$ to $C^{k-2,\alpha}$.
The implicit function theorem gives a unique solution near $1$,
depending continuously on $g$ in $C^{k,\alpha}$. It is positive after
shrinking the neighborhood. Elliptic regularity makes the solution
smooth for smooth $g$. Applying the same argument at higher regularity,
and using uniqueness to identify the solutions, gives continuity in
the smooth topology.

Suppose $\Sc_g\geq-6$. At a maximum point of $u=u_g$,
$\Delta_g u\geq0$, so
\[
 -6u^5=8\Delta_g u+\Sc_g u\geq-6u.
\]
Thus $u\leq1$ everywhere. To compare entropies, lift both metrics
to the universal cover. Since $\widehat g\leq g$, their balls satisfy
$B_g(x,R)\subset B_{\widehat g}(x,R)$. Compactness of $M$ gives
$dV_g\leq C\,dV_{\widehat g}$ for a constant independent of $R$.
Consequently
\[
 \Vol_g B_g(x,R)\leq C\Vol_{\widehat g}B_{\widehat g}(x,R).
\]
Taking logarithms, dividing by $R$, and passing to the limit proves
$h(g)\leq h(\widehat g)$. Finally, continuity of the normalization
allows its image to lie in any prescribed neighborhood of $g_0$.
Thus a bound on the constant-scalar-curvature metrics implies the
bound under $\Sc_g\geq-6$. The converse follows by restriction
to $\Sc_g=-6$.
\end{proof}

\begin{proof}[Proof of Corollary~\ref{cor:local-entropy-curves}]
Shrink the neighborhood so that the conformal normalization
$\widehat g=u^4g$ in Proposition~\ref{prop:conformal-reduction} lies in
the neighborhood of Proposition~\ref{prop:constrained-local-maximum}.
If $\Sc_g\geq-6$, these propositions give
\[
 h(g)\leq h(\widehat g)\leq2.
\]
To determine equality, put $v=1-u\geq0$. The conformal equation gives
\[
 8\Delta_g v+6u(1+u+u^2+u^3)v=(\Sc_g+6)u\geq0.
\]
The zero-order coefficient is positive. Since $M$ is connected, the
strong maximum principle implies that either $v\equiv0$ or $v>0$
everywhere. In the latter case, compactness gives $\rho=\max_M u<1$.
Then $\widehat g\leq\rho^4g$, so metric comparison and entropy scaling
give
\[
 h(g)\leq\rho^2h(\widehat g)\leq2\rho^2<2.
\]
Thus $h(g)=2$ forces $u=1$, and the equality statement in
Proposition~\ref{prop:constrained-local-maximum} makes $g$ a pullback of
$g_0$. The converse follows from diffeomorphism invariance.
\end{proof}

\begin{remark}\label{rem:logic}
Choose the neighborhood so that $C\|S\|_{H^{22}}<c_0/4$ after fixing
the gauge. By Lemma~\ref{lem:entropy-remainder}, the Taylor remainder
is then bounded by $(c_0/4)\|S\|_{H^1}^2$, uniformly for all $S$
in this neighborhood.
\end{remark}

\section{Codazzi tensors and higher-dimensional saddles}
\label{hd:codazzi-section}

Equality in the TT lower bound characterizes trace-free Codazzi tensors
in every dimension. Their Hessian coefficient is negative for $n=3$
and positive for $n\geq4$.

\begin{corollary}\label{hd:codazzi-sign}
For a nonzero trace-free Codazzi tensor $S$,
\begin{equation*}
 \cH(S,S)=\frac{(n-1)(n-2)}{8}
       \bigl(K_n-\frac{n+2}{n-1}\bigr)\|S^\vee\|^2_{L^2(SM,m_0)}.
\end{equation*}
This is negative for $n=3$ and positive for every $n\geq4$.
For $n=4$,
\[
 \cH(S,S)=\frac1{10}\|S^\vee\|^2_{L^2(SM,m_0)}
        =\frac{1}{120V_0}\|S\|^2_{L^2(M)}.
\]
\end{corollary}
\begin{proof}
Proposition~\ref{hd:bochner} gives $\lambda=n$. Apply
\eqref{hd:main-mode}, \eqref{hd:bottom-value}, and
\eqref{hd:bottom-recurrence}. In dimension four $K_4=32/15$;
\eqref{hd:fiber} gives the second expression.
\end{proof}

\subsection{From bending cohomology to a TT tensor}

\Needspace{8\baselineskip}
\begin{proposition}\label{hd:bending-input}
For each $n\geq3$ there are closed oriented hyperbolic manifolds
$M=\Gamma\backslash\mathbb H^n$ with
\[
 H^1(\Gamma;\mathbb R^{n,1})\ne0,
\]
where $\Gamma$ acts through its standard hyperbolic holonomy.
Here $H^1$ denotes group cohomology with the indicated coefficients.
These examples can be obtained by bending along embedded totally geodesic
hypersurfaces, viewing the holonomy in $SO_0(n+1,1)$.
\end{proposition}
Johnson--Millson prove this using bending and arithmetic constructions
(see \cite{JM}*{Sections~5 and~7}). For a closed hyperbolic manifold of
dimension $n\geq3$, Weil's infinitesimal rigidity theorem gives
$H^1(\Gamma;\mathfrak{so}(n,1))=0$. The decomposition of
$\Gamma$-modules
\[
 \mathfrak{so}(n+1,1)=\mathfrak{so}(n,1)\oplus\mathbb R^{n,1}
\]
therefore identifies infinitesimal deformations of the holonomy in
$SO_0(n+1,1)$ with $H^1(\Gamma;\mathbb R^{n,1})$.
Bending along an embedded two-sided totally geodesic hypersurface
gives a nonzero class in this standard module. Such closed manifolds
exist in every dimension $n\geq3$.

The relation between this cohomology and Codazzi tensors is described
for hyperbolic surfaces by Bonsante--Seppi (see \cite{BS}). We give the
flat-bundle argument in all dimensions needed here.

\begin{theorem}\label{hd:flat-bundle}
On a closed hyperbolic manifold $M=\Gamma\backslash\mathbb H^n$,
\[
 \{\hbox{trace-free Codazzi symmetric two-tensors}\}
 \simeq H^1(\Gamma;\mathbb R^{n,1}).
\]
\end{theorem}
\begin{proof}
The flat bundle associated with the standard holonomy is
$E=TM\oplus\mathbb R$, with connection
\begin{equation*}
 D_X(Y,s)=(\nabla_XY+sX,\ X(s)+g_0(X,Y)).
\end{equation*}
The curvature $-1$ identity gives $D^2=0$, and $D$ preserves the
Lorentz form $|Y|^2-s^2$.  Since the universal cover of $M$ is
contractible, flat de Rham theory gives
\[
 H^1_{\mathrm{dR}}(M;E)
 \simeq H^1(\Gamma;\mathbb R^{n,1}),
\]
where $H^1_{\mathrm{dR}}(M;E)$ denotes the $D$-closed $E$-valued
one-forms modulo $D$-exact forms.  We identify endomorphisms with
two-tensors using $g_0$.  A symmetric Codazzi tensor
$A$ defines a closed form $(A,0)$; we prove that each cohomology
class has a unique such representative with $\tr A=0$.

Let $\alpha(X)=(A(X),b(X))$ be a closed form.
The scalar component of $D(b^\sharp,0)$ is $b$, so subtracting
this exact form gives a representative $(A,0)$, with $A$ renamed.
The scalar and tangent components of its closedness condition give
\[
 g_0(A(X),Y)=g_0(X,A(Y)),\qquad
 (\nabla_XA)(Y)=(\nabla_YA)(X).
\]
Thus $A$ is symmetric and Codazzi.  To remove its trace, observe that
\[
 D(\nabla f,-f)(X)=((\Hess f-f\operatorname{Id})(X),0).
\]
The trace of its tangent component is $-(\Delta+n)f$, since
$\Delta=-\tr\Hess$ on functions.  Solve
$-(\Delta+n)f=\tr A$, using invertibility of $\Delta+n$ on the
closed manifold $M$.  Subtracting this exact form removes the trace.
The new representative is still closed and has scalar component zero,
so it is still symmetric and Codazzi.

To prove uniqueness, suppose a trace-free representative is exact:
$(A,0)=D(Y,s)$.  Its scalar component gives $Y=-\nabla s$,
so $A=-\Hess s+s\operatorname{Id}$.  Taking the trace gives
$(\Delta+n)s=0$, hence $s=0$ and $A=0$.
This proves the isomorphism.
\end{proof}

\begin{proof}[Proof of Corollary~\ref{hd:saddle-intro}]
Choose a nonzero trace-free Codazzi tensor $S^+$. Its Hessian is positive
by Corollary~\ref{hd:codazzi-sign}. Choose any TT eigentensor $S^-$
with eigenvalue greater than $\lambda_*(n)$; its Hessian is negative.
Lemma~\ref{lem:integrability} integrates both tensors to constrained curves.
Their first entropy variations vanish, and therefore
\[
 h(g_s^\pm)=n-1+\frac{s^2}{2}\cH(S^\pm,S^\pm)+o(s^2).
\]
This gives the two strict inequalities for sufficiently small nonzero
$s$. Proposition~\ref{hd:bending-input} and
Theorem~\ref{hd:flat-bundle} provide such manifolds in every dimension
$n\geq4$. The curves are smooth and negatively curved near zero.
\end{proof}

The entropy-increasing curves have constant scalar curvature and
leave the space of hyperbolic metrics, since
\eqref{eq:ricci-variation} gives
\[
 D\bigl(\Ric+(n-1)g\bigr)_{g_0}(S)
 =\frac12(\Delta-2)S=\frac{n-2}{2}S
\]
on a Codazzi tensor, which is nonzero for $n\geq3$.
The constrained curve has positive volume acceleration
\eqref{eq:volume-second}. Volume normalization contributes the second,
negative term in \eqref{eq:main-hessian}.

For a manifold whose lowest TT eigenvalue is larger than $\lambda_*(n)$,
the Hessian is instead negative definite. If the lowest eigenvalue equals the threshold, there is a
finite-dimensional nullspace where the second variation does not decide the
behavior. A local maximum would require a remainder estimate in addition to
these second-variation inequalities.

\end{document}